\documentclass[10pt]{article}
\usepackage{amsthm, amsmath, amssymb}
\usepackage{graphicx}
\usepackage[normalem]{ulem}
\usepackage{float}
\usepackage{hyperref}
\usepackage[left=3.7cm,right=3.7cm,top=2.5cm,bottom=3cm]{geometry}

\theoremstyle{plain}
\newtheorem{theorem}{Theorem}[section]

\newtheorem{proposition}[theorem]{Proposition}

\newcommand{\Z}{{\mathbb Z}}
\newcommand{\R}{{\mathbb R}}
\newcommand{\C}{{\mathbb C}}
\newcommand{\DTCWT}{DT-$\C$WT}

\title{The Near Shift-Invariance of the Dual-Tree Complex Wavelet Transform Revisited\\[0.2cm]
{\footnotesize This article has been accepted for publication in the}\\[-0.4cm] {\footnotesize ``Journal of Mathematical Analysis and Applications'' by Elsevier\footnote{\scriptsize This is a preprint version, prepared for posting on ArXiv. It incorporates corrections made by the
authors in response to comments by reviewers. It does not incorporate any subsequent editing changes by Elsevier made in preparation for the final, published version. It is the authors' understanding that,
under the rules posted on \url{http://www.elsevier.com/about/open-access/open-access-policies/article-posting-policy}, this posting does not infringe on subsequent copyright transfer to Elsevier. }}\\[-0.4cm]
{\footnotesize DOI: \url{http://dx.doi.org/10.1016/j.jmaa.2012.01.010}}}

\author{\small Adriaan Barri$^{\text{a},\text{b}}$,\hspace*{0.68cm} Ann Dooms$^{\text{a},\text{b}}$,\hspace*{1.05cm} Peter Schelkens$^{\text{a},\text{b}}$\\[-0.15cm]
\scriptsize\url{abarri@etro.vub.ac.be}\hspace*{0.5cm}\url{adooms@etro.vub.ac.be}\hspace*{0.5cm} \url{pschelke@etro.vub.ac.be}}

\date{\begin{flushleft}\hspace*{1.2cm}\footnotesize\noindent $^{\text{a}}$ Dept. of Electronics and Informatics (ETRO), Vrije Universiteit Brussel, Belgium\end{flushleft}
\vspace*{-0.615cm}
\begin{flushleft}\footnotesize\hspace*{1.2cm}$^{\text{b}}$ iMinds VZW, Gaston Crommenlaan 8 (box 102), B-9050 Ghent, Belgium.\end{flushleft}}

\begin{document}

\maketitle
\vspace*{-0.7cm}
\begin{abstract}
The dual-tree complex wavelet transform ({\DTCWT}) is an enhancement of the conventional discrete wavelet transform (DWT) due to a higher degree of shift-invariance and a greater directional selectivity, finding its applications in signal and image processing. This paper presents a quantitative proof of the superiority of the {\DTCWT} over the DWT in case of modulated wavelets.

\vspace*{0.4cm}

\end{abstract}

\section{Introduction}

Wavelet transforms  provide a convenient technique to perform a multiresolution analysis of finite-energy signals. The most popular instance of a wavelet transform is the critically sampled discrete wavelet transform (DWT),  which is an invertible transform that permits sparse signal decompositions at a low computational cost \cite{mallat}.

The DWT has been successfully employed in many applications, including image compression \cite{skodras}, noise reduction \cite{donoho} and speech recognition \cite{favero}. However, in the area of statistical signal processing, the DWT has proven to be less effective \cite{coifman, pesquet}. This is mainly due to the
high \emph{translation sensitivity} of the DWT: small shifts in the input signal may completely change the wavelet coefficient pattern. As a consequence, algorithms based on the DWT need to recognize and understand a wide variety of different wavelet patterns.

One way to address the shift-variance problem is to relax the critical sampling criterion of the DWT. In \cite{sarraf}, an overcomplete version of the DWT is proposed, which is most easily implemented by the ``\`a trous'' algorithm. A generalization of this algorithm is described in \cite{bradley}.  Note that this approach is computationally intensive and produces highly redundant output information, which limits its applicability. Nevertheless, since the output of the ``\`a trous'' algorithm can be computed directly from the critically sampled DWT, it is readily applied in DWT-based image and video coding systems \cite{andreopoulos1,andreopoulos2}.

In \cite{simoncelli1, simoncelli2}, Simoncelli et al. introduce the steerable pyramid, an alternative decomposition method based on Laplacian pyramids and steerable filters that achieves approximate shift-invariance. Furthermore, the steerable pyramid also gives a better directional selectivity when analyzing two-dimensional signals, which simplifies the extraction of  geometric features in images.

Another way to improve the shiftability of the DWT is by simultaneously employing two real DWT channels that form an approximate Hilbert Transform pair.  By combining the corresponding coefficients of the first and second DWT into complex-valued coefficients, we obtain a new transform, which  is called the \emph{dual-tree complex wavelet transform} ({\DTCWT}).  Compared to the steerable pyramid, the {\DTCWT} provides a better directional selectivity while having a lower redundancy factor of $2^d$ for $d$-dimensional signals. A more elaborate discussion on the design and use of the {\DTCWT} can be found in \cite{selesnick2}.

The near shift-invariance property of the {\DTCWT} has been extensively studied over the last decade \cite{chaud1, kingsbury1, kingsbury2, ozkaramanli, selesnick1,yu}.  Recently, Chaudhury and Unser \cite{chaud1} deduced an amplitude-phase representation for dual-tree complex wavelet transforms that involve \emph{modulated wavelets}, linking the multiresolution framework of the wavelet components to the frequency decomposition through Fourier Analysis. This representation provided new insights into the improved shiftability of the {\DTCWT}.

In this contribution, we build on their findings by introducing a more formal description of the DWT translation sensitivity, which will allow us to better explain the superiority of the {\DTCWT}. We finish  with a study on the decaying rate of the {\DTCWT} shift error when the translation parameter tends to zero in case of orthonormal wavelet systems.

\section{Preliminaries}
\label{sec:preliminaries}

We now introduce some definitions and notation needed to state our results in the following sections.

Given two signals $f$ and $g$ in $L^2(\R)$, we define their inner product by
$$\bigl<f,g\bigr> = \int_{\R} f(x) \overline{g(x)} dx,$$
where the bar indicates complex conjugation.
The Fourier transform of $f$ is given by
$$\hat{f}(\xi) = \int_{\R} f(x) e^{-i\xi x}dx$$
whereas the Hilbert transform $\mathcal{H}$ (HT) is characterized by the relation
$$\widehat{\mathcal{H}f}(\xi) = -i \text{sign}(\xi)\hat{f}(\xi).$$
The Hilbert transform is orthogonal to the signal, commutes with translations and positive dilatations, and $\mathcal{H}^{-1} = -\mathcal{H}$.
The translation-dilatation operator $\Xi_{j,k}$ on $\psi\in L^2(\R)$ is defined by
$$\Xi_{j,k}[\psi] = 2^{j/2}\psi(2^j\cdot-k)=\psi_{j,k}.$$

Let $\{\psi_{j,k}\}_{j,k\in\Z}$ and $\{\psi'_{j,k}\}_{j,k\in\Z}$ be two real-valued bi-orthogonal wavelet systems that form
a \emph{Hilbert transform pair}, i.e. $\psi' = \mathcal{H}\psi$. We define the wavelet coefficients of $f$ with respect to these wavelet systems by
$$a_j[k] = \bigl< f, \psi_{j,k}\bigr>\quad \text{and}\quad b_j[k] = \bigl< f, \psi'_{j,k}\bigr>$$
for every $j,k\in\Z$.
 These equations yield the following two different wavelet identities:
\begin{align*}
f = \sum_{j,k\in\Z} a_j[k] \tilde{\psi}_{j,k}\quad \text{and}\quad
f = \sum_{j,k\in\Z} b_j[k] \tilde{\psi}'_{j,k},
\end{align*}
where $\tilde{\psi}_{j,k}$ and $\tilde{\psi}'_{j,k}$ represent the dual wavelets of $\psi_{j,k}$ and $\psi'_{j,k}$ respectively.

We now introduce the \emph{complex wavelets}
\begin{align*}
\Psi_{j,k} = \frac{\psi_{j,k}+i\psi'_{j,k}}{2}\quad \text{and}\quad \tilde{\Psi}_{j,k} = \frac{\tilde{\psi}_{j,k}+i\tilde{\psi}'_{j,k}}{2}.
\end{align*}
The \emph{{\DTCWT} coefficients} are then given by
\begin{align*}
c_j[k] &= \bigl<f,\Psi_{j,k}\bigr>\\
 &= \tfrac{1}{2}(a_j[k]-ib_{j}[k])
\end{align*}
for every $j,k\in\Z$.

Recall that for dyadic wavelet transforms, the level $j$ coefficients of a shifted signal  $f(\cdot + s)$ with $s=2^{-j}m$, $m\in\Z$, can be easily predicted from the coefficients of the reference signal. In fact,
\begin{align*}
c_j^s[k] &= \bigl<f(\cdot + 2^{-j}m), \Psi_{j,k} \bigr>\\
&= \bigl< f, \Psi_{j,k}(\cdot-2^{-j}m)\bigr>\\
&= \bigl< f, \Psi_{j,k+m}\bigr>\\
&= c_j[k+m].
\end{align*}
This well-known property can be adapted for arbitrary shifts $s$ by decomposing $s$ into a dyadic number $2^{-j}m$ and some remainder $h$ with $|h|<2^{-j}$:
$$s=2^{-j}m + h.$$
Then
\begin{align*}
c_j^s[k] &= \bigl< f(\cdot +s), \Psi_{j,k} \bigr>\\
&= \bigl< f(\cdot + h), \Psi_{j,k+m} \bigr>\\
&= c_j^{h}[k+m].
\end{align*}
The adjusted shift error
$$|c_j[k+m]-c_j^s[k]| = |c_j[k+m]-c_j^{h}[k+m]|$$
is in general much smaller than the original shift error $|c_j[k]-c_j^s[k]|$.
In order to further reduce the shift error $|c_j[k+m]-c_j^{h}[k+m]|$, we will perform a phase change of $c_j[k+m]$ over an angle $\phi_h$ that partially compensates for the small shift $h$ ($|h|<2^{-j}$), so that
$$c_j^s[k] = c_j^h[k+m] \approx  e^{i\phi_h}c_j[k+m].$$

As suggested in \cite{chaud1}, we  make the assumption that the involved wavelet $\psi$ is \emph{modulated}. That is,
\begin{align*}
\psi(x) = w(x)\cos(\omega_0x+\xi_0)
\end{align*}
for $\omega_0, \xi_0>0$ where the localization window $w$ is bandlimited to $[-\Omega,\Omega]$ for some $\Omega<\omega_0$. Examples of modulated wavelets are the Shannon and Gabor wavelets.
As the orthonormal spline, resp. $B$-spline, wavelets resemble the Shannon, resp. Gabor, wavelet, they can be seen as a kind of modulated wavelets.
Using the \emph{Bedrosian identity} (see \cite{chaud1}), one can show that
\begin{align*}
\psi'(x) = w(x)\sin(\omega_0x + \xi_0).
\end{align*}
In this way, we obtain the
 identity
\begin{align*}
\Psi(x) = \frac{e^{i\xi_0}}{2} w(x)e^{i\omega_0x}.
\end{align*}

In order to examine the near shift-invariance of the {\DTCWT} based on these modulated wavelets, we thus need to minimize the error
\begin{align*}
|e^{i\phi_h}c_j[k]-c_j^h[k]|
\end{align*}
for some well-chosen angle $\phi_h\in[-\pi,\pi[$.
This \emph{phase-compensated shift error} will be compared to the ``real'' shift errors $|a_j[k]-a_j^h[k]|$ and $|b_j[k]-b_j^h[k]|$ in Section \ref{sec:phasecompensation}. The attained results (summarized in Theorem \ref{thm:phasecompensation}) suggest to take $\phi_h = 2^j\omega_0 h$, which accords with the conclusions drawn by Chaudhury in \cite[p. 127-128]{chaud3}.

We have empirically verified that $c_j^h[k]\approx e^{i2^j\omega_0h}c_j[k]$ for small shifts $h$   using the {\DTCWT} software provided by the authors of \cite{chaud2}.
In our experiments, we put $\psi=\psi_{4.5}^{8}$ and $\psi' = \psi_{5}^{8}$, the fractional B-spline wavelets of degree $\alpha=8$ and with shift parameters $\tau=4.5$,  $\tau=5$ respectively \cite{unser}. These wavelets are known to be approximately modulated. In fact, one can observe that
$$\psi(x) \approx  w(x)\cos(5.3x+5.2),$$
where $w = \sqrt{\psi^2 + \psi'^2}$. It is proven in \cite{chaud2} that $\psi$ and $\psi'$ form a Hilbert Transform pair; therefore, they determine a {\DTCWT}.

Our test signal consists of 512 uniform samples from a block function $f:[0,1]\rightarrow \R$ (see Figure \ref{fig:testsetting}), which is shifted  one place to the left. The resulting signal corresponds to the function $f^h = f(\cdot + 1/512)$.

Figure \ref{fig:testresults} compares the phase-compensated shift error $|e^{i2^j\omega_0h}c_j[k]-c_j^h[k]|$ with the optimal shift error
$$\bigl||c_j[k]|-|c_j^h[k]|\bigr|= \min_{\phi_h\in]-\pi,\pi]}|e^{i\phi_h}c_j[k]-c_j^h[k]|.$$ The graphs on the right are zoomed out so that the shift errors of the real and imaginary wavelet components can be included.
The following naming conventions are used for the considered shift-errors:
\begin{align*}
\texttt{complex-optimal} &= \bigl||c_j[k]|-|c_j^h[k]|\bigr|\\
\texttt{complex-phasecomp} &= \bigl|\text{exp}(i*2^j*5.3/512)c_j[k] - c_j^h[k]\bigr|\\
\texttt{real} &= 2|\text{Re}(c_j[k])-\text{Re}(c_j^h[k])|\\
\texttt{imag} &= 2|\text{Im}(c_j[k])-\text{Im}(c_j^h[k])|
\end{align*}

The test results show that the phase-compensated shift error is significantly smaller than the shift error of the real and imaginary wavelet components. Moreover, there is a very good fit between the phase-compensated and the optimal shift error. Therefore, we can conclude that the derived phase-compensation for the {\DTCWT} coefficients is near-optimal under small translations.

\begin{figure}[H]
\centering
\hspace*{-0.7cm}
 \includegraphics[height=5.8cm]{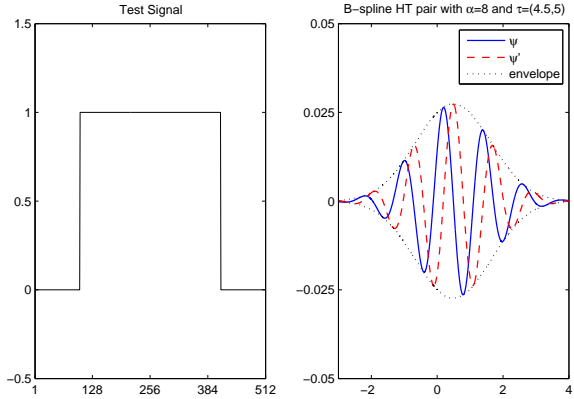}
\caption{Graphical representation of the test signal and the selected fractional $B$-spline wavelet pair.}
\label{fig:testsetting}
\end{figure}

\begin{figure}[H]
\centering
\includegraphics[height=12.5cm]{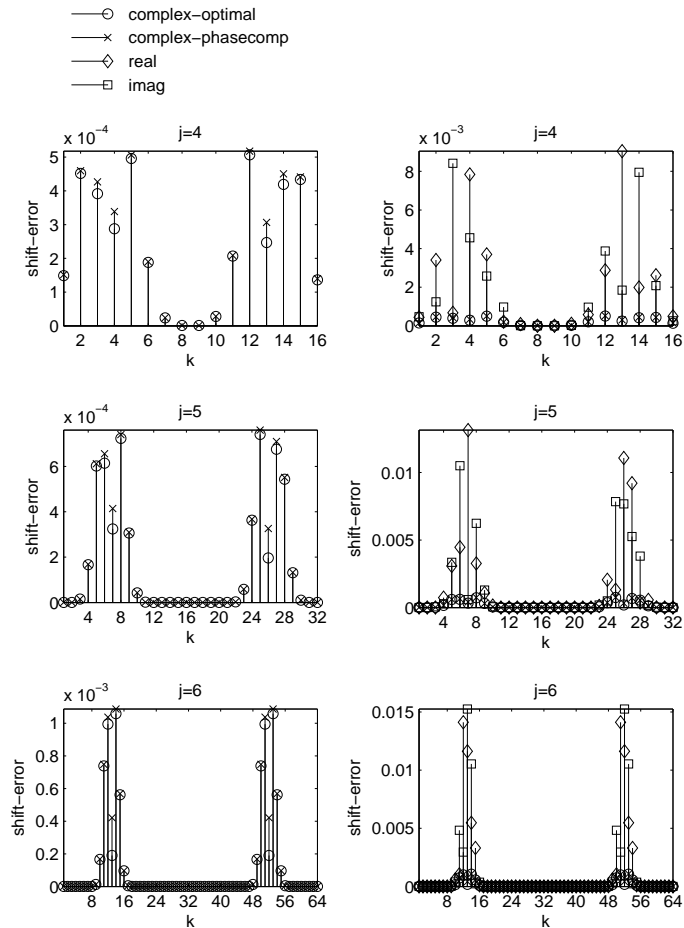}
\caption{Graphs showing the shift-errors of the {\DTCWT} and the real and imaginary wavelet components. The plots on the left compare the phase-compensated error with the optimal shift-error. The plots on the right are zoomed out so that they can show the shift errors of the real and imaginary wavelet components.}
\label{fig:testresults}
\end{figure}

\section{On the phase-compensated shift error  of the Dual-Tree Complex Wavelet Transform}
\label{sec:phasecompensation}

Let $f$ be a real-valued function in $L^2(\R)$ with {\DTCWT} coefficients $c_j[k],j,k\in\Z$, based on two \emph{modulated} wavelets $\psi$ and
$\psi'=\mathcal{H}\psi$. Denote the translates of $f$ over some real number $h$  by $f^h = f(\cdot + h)$. The {\DTCWT} coefficients of $f^h$ are given by
$$c_j^h[k] = \frac{1}{2}(a_j^h[k]-ib_j^h[k]),\quad j,k\in\Z.$$
As deduced in Section \ref{sec:preliminaries}, to study the shift error for the level $j$ coefficients, it is enough to look at $|h|<2^{-j}$.

In this section, we work towards a proof of the following theorem.

\begin{theorem}[Phase Compensation for the {\DTCWT} shift error]
\label{thm:phasecompensation}
Let $c_j[k]$, $j,k\in\Z$, be the coefficients of a real-valued function $f$ in $L^2(\R)$ with respect to a {\DTCWT} decomposition for which  the involved wavelet $\psi$ is modulated.  When $h$ is small, we have the approximate identities
\begin{align*}
\frac{|e^{i2^j\omega_0h}c_j[k] - c_j^h[k]|}{  |a_j[k]-a_j^h[k]| } \approx 0
\quad \text{and} \quad
\frac{|e^{i2^j\omega_0h}c_j[k] - c_j^h[k]|}{|b_j[k]-b_j^h[k]| } \approx 0,
\end{align*}
proving that the phase-compensated error is negligible in relation to   the shift errors of the real and imaginary wavelet components.
\end{theorem}

We express the \emph{translation sensitivity} of the DWT coefficients $a_j[k]$ and $b_j[k]$ by postulating that
\begin{align}
\label{translationsensitivity}
\inf_{|h|<2^{-j}}\frac{|a_j[k]-a_j^h[k]|}{|h\, a_j[k]|} = B_a\quad \text{and}\quad \inf_{|h|<2^{-j}}\frac{|b_j[k]-b_j^h[k]|}{|h\, b_j[k]|} = B_b
\end{align}
for some values $B_a$ and $B_b$ significantly larger than zero.

\medskip

In the next proposition, we introduce the ratio $R_h$, which relates the phase-compensated shift error  $|e^{i\phi_h} c_j[k] - c_j^h[k]|$ to the shift errors of the real and imaginary wavelet components $|a_j[k]-a_j^h[k]|$ and $|b_j[k] - b_j^h[k]|$.

\begin{proposition}
\label{prop:ratiorealcomplex}
Let $\phi_h\in [-\pi,\pi[$ for every $h\in\R$ with $|h|<2^{-j}$.
Define
\begin{align}
\label{Rh}
R_h = \frac{e^{i\phi_h}a_j[k]-a_j^h[k]}{e^{i\phi_h}b_j[k]-b_j^h[k]}.
\end{align}
Suppose that the constraint in (\ref{translationsensitivity}) holds for $B_a$ and $B_b$ significantly larger than zero. Then
\begin{align}
\label{DTCWTboundreal}
\frac{|e^{i\phi_h}c_j[k]-c_j^h[k]|}{|a_j[k]-a_j^h[k]|} &\leq (1+\Phi/B_a)\frac{|1-i/R_h|}{2}
\intertext{and}
\label{DTCWTboundimag}
\frac{|e^{i\phi_h}c_j[k]-c_j^h[k]|}{|b_j[k]-b_j^h[k]|} &\leq (1+\Phi/B_b)\frac{|R_h-i|}{2},
\end{align}
where
$\Phi = \sup_{|h|<2^{-j}}\bigl|(e^{i\phi_h}-1)/h\bigr|$.

Note that $\Phi$ is finite if and only if $\limsup_{h\rightarrow 0} |\phi_h/h| < \infty$.
\end{proposition}
\begin{proof}
We prove (\ref{DTCWTboundreal}); the proof of (\ref{DTCWTboundimag}) is similar.
Observe that
\begin{align*}
|e^{i\phi_h}c_j[k] - c_j^h[k]|
&= \frac{1}{2}\Bigl|(e^{i\phi_h}a_j[k]-a_j^h[k] )- i(e^{i\phi_h}b_j[k]-b_j^h[k])\Bigr|\\
&= \frac{1}{2}\Bigl|(1-i/R_h)(e^{i\phi_h}a_j[k] - a_j^h[k])\Bigr|\\
&= \frac{1}{2}|1-i/R_h|\, |e^{i\phi_h}a_j[k] - a_j^h[k]|.
\end{align*}
Dividing $|e^{i\phi_h}c_j[k]-c_j^h[k]|$ by $|e^{i\phi_h}a_j[k] - a_j^h[k]|$ results in
\begin{align*}
\Biggl|\frac{e^{ i\phi_h}c_j[k]- c_j^h[k]}{ e^{i\phi_h}  a_j[k]-a_j^h[k]}\Biggr| &= \frac{|1- i/R_h|}{2}.
\end{align*}
On the other hand, we have
$$\frac{|e^{i\phi_h}-1|\, |a_j[k]|}{|a_j[k]-a_j^h[k]|} \leq \Phi/B_a,$$
which implies that
\begin{align*}
\frac{|e^{i\phi_h}a_j[k] -a_j^h[k]|}{|a_j[k]-a_j^h[k]|}
&\leq \frac{|a_j[k]-a_j^h[k]| + |e^{i\phi_h}-1|\, |a_j[k]|}{|a_j[k]-a_j^h[k]|}\\
&\leq  1 + \Phi/B_a.
\end{align*}
Hence,
\begin{align*}
\frac{|e^{i\phi_h}c_j[k]-c_j^h[k]|}{|a_j[k]-a_j^h[k]|} &=\frac{|e^{i\phi_h}a_j[k] -a_j^h[k]|}{|a_j[k]-a_j^h[k]|} \times \frac{|1- i/R_h|}{2}  \\
&\leq (1 + \Phi/B_a)\frac{|1- i/R_h|}{2}. \tag*{\qedhere}
\end{align*}
\end{proof}

\noindent This shows that the phase-compensated shift error $|e^{i\phi_h}c_j[k]-c_j^h[k]|$ becomes  smaller as the ratio $R_h$ approaches to $i$.

\medskip

The perturbed coefficients $c_j^h[k]$ can be expressed more explicitly as
\begin{align*}
c_j^h[k] &= \left<f(\cdot + h)   , \Psi_{j,k}\right>\\
&= \left<f, \Psi_{j,k}(\cdot - h)\right>\\
&= \frac{e^{-i\xi_0}}{2} \int_{\R} f(x) \Xi_{ j,k}[w(x-2^jh)e^{-i\omega_0(x-2^jh)}]dx.
\end{align*}
Observe that both the
sinusoid and the localization window of the modulated wavelet contribute to the perturbation of $c_j^h[k]$. Their individual roles on the {\DTCWT} shiftability can be described using the variables
\begin{align*}
E_h = \int_\R f(x) \Xi_{j,k}\Bigl[w(x)\bigl(e^{-i\omega _0x} - e^{-i\omega_0 (x-2^jh)}\bigr)\Bigr]dx
\end{align*}
and
\begin{align*}
W_h = \int_{\R} f(x)\Xi_{j,k}\Bigl[\bigl(w(x)-w(x-2^jh)\bigr)e^{- i\omega_0 x}\Bigr]dx.
\end{align*}

\begin{proposition}
\label{prop:limWhEh}
If the localization window $w$ is differentiable, then $|W_h/E_h|$ converges as $h\rightarrow 0$. More precisely, we have
the identity
\begin{align*}
\lim_{h\rightarrow 0} |W_h/E_h| &= \frac{1}{2\omega_0|c_j[k]|}\biggl|\int_{\R}f(x)\Xi_{j,k}[\tfrac{dw}{dx}(x)e^{-i\omega_0x}]dx
\biggr|.
\end{align*}
\end{proposition}

\noindent By the Cauchy-Schwarz inequality, we get that 
\begin{align*}
\lim_{h\rightarrow 0} |W_h/E_h| \leq \frac{1}{2\omega_0|c_j[k]|}\bigl\|f\bigr\|_2 \bigl\|\tfrac{dw}{dx}\bigr\|_2.
\end{align*}

This expression reveals three important parameters that influence the asymptotic behavior of $|W_h/E_h|$ as $h\rightarrow 0$: firstly, the frequency $\omega_0$ of the modulated wavelet; secondly, the $L^2$-norm of $\frac{dw}{dx}$; and thirdly, the significance of the {\DTCWT} coefficient $\frac{|c_j[k]|}{\|f\|_2}$.

The first two parameters are close to zero for most modulated wavelets,  because the overall frequency $\omega_0$ is very high compared to the slowly varying localization window $w$, which is for example illustrated in Figure \ref{fig:testsetting}. We can thus conclude that $|W_h/E_h|\approx 0$ for all significant coefficients $c_j[k]$ when $h$ is small.

Clearly,
\begin{align}
\label{Eheq}
|E_h| &=   2|e^{i2^{j}\omega_0 h}-1|\,|c_j[k]|.
\intertext{By defining $\alpha_h$ as the unique angle in $[0,\pi[$ that satisfies}
|W_h| &= 2|e^{i\alpha_h}-1|\, |c_j[k]|,
\label{Wheq}
\end{align}
we arrive at
\begin{align}
\label{Wheheq}
|W_h/E_h| &= \frac{|e^{i\alpha_h}-1|}{|e^{i2^{j}\omega_0 h}-1|}.
\end{align}

In the next proposition, we show that the ratio $R_h$ corresponding to the phase-compensation $\phi_h = 2^j\omega_0h + \text{sign}(h)\alpha_h$ is approximately equal to $i$  when $h$ is small and $|W_h/E_h|\approx 0$.

\begin{proposition}
\label{prop:phi}
Let $\beta_h\in ]-\pi, \pi]$, such that $\beta_h = 2^{j+1}\omega_0h$ modulo $2\pi$. Define
\begin{align}
\label{phi}
\phi_h = 2^j\omega_0h + \text{sign}(\beta_h)\alpha_h.
\end{align}
If $    \alpha_h < \pi - |\beta_h|$ and $|W_h/E_h|<1$, then
\begin{align}
R_h=\frac{e^{i\phi_h}a_j[k]-a_j^h[k]}{e^{i\phi_h}b_j[k]-b_j^h[k]} = i\, \frac{1+K_h}{1-K_h},
\end{align}
where $K_h$ is a complex number that satisfies
\begin{align*}
|K_h| \leq \frac{2|W_h/E_h|}{|e^{i2^j\omega_0 h} +1|-|W_h/E_h|}.
\end{align*}
\end{proposition}
\begin{proof}
We prove the proposition for $\beta_h>0$.
Define $\Delta = e^{i\phi_h}f - f^h$. Since $\psi = \Psi + \overline{\Psi}$ and $\psi' = -i(\Psi - \overline{\Psi})$ by construction, we obtain that
\begin{align*}
e^{i\phi_h}a_j[k]-a_j^h[k] &= \bigl<\Delta,\Psi_{j,k}\bigr> + \bigl< \Delta, \overline{\Psi}_{j,k}\bigr>
 \intertext{and}
e^{i\phi_h}b_j[k]-b_j^h[k] &= i\bigl<\Delta,\Psi_{j,k}\bigr> - i\bigl< \Delta, \overline{\Psi}_{j,k}\bigr>.
\end{align*}
This shows that $R_h=i(1+K_h)/(1-K_h)$, with
\begin{align*}
K_h &= \frac{\bigl<\Delta, \Psi_{j,k}\bigr>}{\bigl< \Delta, \overline{\Psi}_{j,k}\bigr>}.
\end{align*}
A simple calculation gives
\begin{align*}
\bigl< \Delta, \Psi_{j,k} \bigr> &=  e^{i\phi_h}\bigl<f,\Psi_{j,k}\bigr>-\bigl<f,\Psi_{j,k}(\cdot-h)\bigr>\\
&= \int_{\R} f(x) \Xi_{j,k}\bigl[e^{i\phi_h}\overline{\Psi}(x)-\overline{\Psi}(x-2^jh)\bigr]dx\\
&= \frac{e^{-i\xi_0}}{2}\int_{\R} f(x)  \Xi_{j,k}\Bigl[e^{-i\omega_0x}\bigl[e^{i\phi_h}w(x)-e^{i2^j\omega_0h}w(x-2^jh)\bigr]\Bigr]\\
&= \frac{e^{i(2^j\omega_0h-\xi_0)}}{2} W_h + \bigl(e^{i\phi_h}-e^{i2^j\omega_0h}\bigr)c_j[k] .
\end{align*}
Similarly,
\begin{align*}
\bigl< \Delta, \overline{\Psi}_{j,k} \bigr> &=   \frac{e^{i(\xi_0-2^j\omega_0h)}}{2} \overline{W_h} +\bigl(e^{i\phi_h}-e^{-i2^j\omega_0h}\bigr)\overline{c_j[k]}.
\end{align*}
Hence,
\begin{align*}
|K_h| &\leq \frac{\bigl|W_h\bigr| + 2\bigl|e^{i(\phi_h-2^j\omega_0h)}-1\bigr|\, \bigl|c_j[k]\bigr|}{\Bigl| \bigl|W_h\bigr| - 2\bigl|e^{i(\phi_h+2^j\omega_0h)}-1\bigr|\, \bigl|c_j[k]\bigr|\Bigr|}.
\intertext{The substitution $\phi_h = 2^j\omega_0h + \alpha_h$ gives}
|K_h| &\leq \frac{\bigl|W_h\bigr| + 2\bigl|e^{i\alpha_h}-1\bigr|\, \bigl|c_j[k]\bigr|}{\Bigl| \bigl|W_h\bigr| - 2\bigl|e^{i(2^{j+1}\omega_0h+\alpha_h)}-1\bigr|\, \bigl|c_j[k]\bigr|\Bigr|}.
\end{align*}
Since
\begin{align*}
|e^{i(2^{j+1}\omega_0h+\alpha_h)}-1|^2 &= 2-2\cos(\beta_h+  \alpha_h)\\
&\geq 2-2\cos(\beta_h)\\
&= |e^{i2^{j+1}\omega_0h}-1|^2
\intertext{we have}
\frac{1}{\bigl|e^{i(2^{j+1}\omega_0h+\alpha_h)}-1\bigr|} &\leq  \frac{1}{\bigl|e^{i2^{j+1}\omega_0h}-1\bigr|}\\
&= \frac{1}{\bigl|e^{i2^j\omega_0h}-1\bigr|\, \bigl|e^{i2^j\omega_0 h}+1\bigr|}.
\end{align*}
This gives us, in combination with Equation (\ref{Wheq}),
\begin{align*}
\frac{\bigl|W_h\bigr|}{\bigl|e^{i(2^{j+1}\omega_0h + \alpha_h)}-1\bigr|}&\leq 2\frac{\bigl|W_h/E_h\bigr|}{\bigl|e^{i2^j\omega_0 h}+1\bigr|}|c_j[k]|.
\intertext{On the other hand,}
\frac{\bigl|W_h\bigr|}{\bigl|e^{i\alpha_h}-1\bigr|}&\leq 2|c_j[k]|.
\end{align*}
We arrive at the bound
\begin{align*}
|K_h|
&\leq
\frac{4\bigl|c_j[k]\bigr|\, \bigl|W_h/E_h\bigr|}{2\bigl|c_j[k]\bigr|\, \bigl|e^{i2^j\omega_0h}+1\bigr| - 2\bigl|c_j[k]\bigr|\, \bigl|W_h/E_h\bigr|}\\
&= \frac{2|W_h/E_h|}{|e^{i2^j\omega_0 h} +1|-|W_h/E_h|}.
\end{align*}
This finishes the proof.\end{proof}

The phase-compensation $\phi_h$ proposed in (\ref{phi}) may be hard to determine in practice. Instead,  when $h$ is small and $|W_h/E_h|\approx 0$, we can put $\phi_h = 2^j\omega_0h$.
Indeed, Formula (\ref{Wheheq}) implies that
\begin{align*}
\lim_{h\rightarrow 0} |W_h/E_h| &= \frac{1}{2^j\omega_0} \lim_{h\rightarrow 0} \frac{\alpha_h}{|h|}.
\end{align*}
As a consequence,
\begin{align}
\label{phiid}
\phi_h/2^j\omega_0 h =1 + sign(h)\alpha_h/2^j\omega_0h\approx 1+ |W_h/E_h|\approx 1.
\end{align}
Figure \ref{fig:Rh} plots the ratio $R_h$ for a number of {\DTCWT} coefficients, using the same test configuration as described in Section \ref{sec:preliminaries}. We see that  the values of $R_h$ are indeed close to $i$. Also note that the accuracy of the approximation depends on the magnitude of the {\DTCWT} coefficients. This phenomenon is to be expected, as the size of $|W_h/E_h|$ increases when the significance of the coefficient $c_j[k]$ decreases (see Proposition \ref{prop:limWhEh}).

\begin{proof}[Proof of Theorem \ref{thm:phasecompensation}]
Let $\phi_h=2^j\omega_0 h$. As noted after Proposition \ref{prop:limWhEh}, we may assume that $|W_h/E_h|\approx 0$. Then (\ref{phiid}) shows that this instance of $\phi_h$ is equivalent to the one specified in (\ref{phi}). Hence, Proposition \ref{prop:phi} is applicable, resulting in $R_h\approx i$. Substitution  into (\ref{DTCWTboundreal}) and (\ref{DTCWTboundimag}) of Proposition \ref{prop:ratiorealcomplex} reveals that  the ratios $\frac{|e^{i2^j\omega_0h}c_j[k] - c_j^h[k]|}{|a_j[k]-a_j^h[k]|}$ and $\frac{|e^{i2^j\omega_0h}c_j[k]- c_j^h[k]|}{|b_j[k]-b_j^h[k]|}$ are both approximately zero. This is exactly what we needed to prove.
\end{proof}

\begin{figure}
\centering
\includegraphics[height=6.7cm]{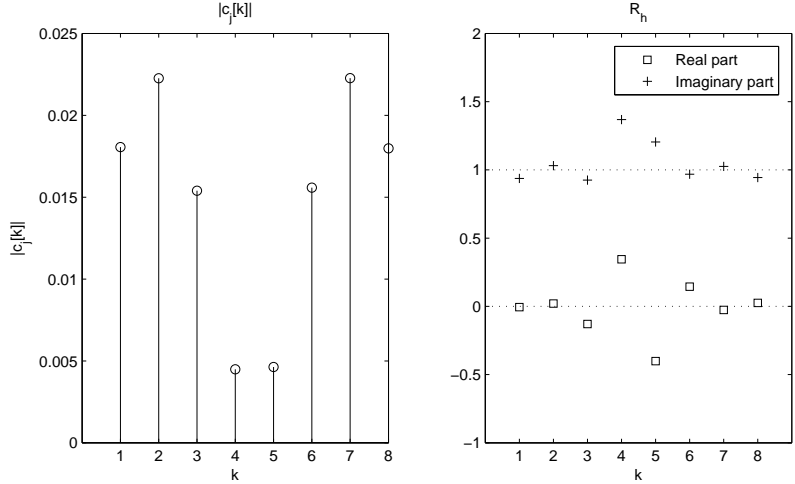}
\caption{Plot of  the ratio $R_h$ for $h=1/512$ and $j=3$. Observe that $R_h$ differs more from $i$ for $k=4,5$.  }
\label{fig:Rh}
\end{figure}

\section{On the decaying rate of the phase-compensated shift error}

In this section, we focus on the dual-tree complex wavelet transform for which the wavelet systems $\{\psi_{j,k}\}_{j,k\in\Z}$ and $\{\psi_{j,k}'\}_{j,k\in\Z}$ are both \emph{orthonormal} and \emph{modulated}. In order to estimate the prediction error $c_j^h[k] \approx e^{i2^j\omega_0h} c_j[k]$, we will introduce a new quantitative bound for the phase-compensated shift error $|e^{i2^j\omega_0h}c_j[k] - c_j^h[k]|$. Moreover, this bound allows us to describe the decaying rate of the phase-compensated shift error as $h\rightarrow 0$.

In \cite{chaud1}, the \emph{fractional Hilbert Transform} (fHT) operator is introduced in order to deduce an amplitude-phase representation of the {\DTCWT}.
The fHT corresponding to the real-valued shift $\tau$ is defined as
$$\mathcal{H}_{\tau} = \cos(\pi\tau)\mathcal{I} - \sin(\pi\tau)\mathcal{H},$$
where $\mathcal{I}$ is the identity operator. Note that for $\tau=-1/2$, we retrieve the original
Hilbert transform operator. Moreover, $\mathcal{H}_{\tau}[\cos(\omega_0 x)] = \cos(\omega_0 x + \pi\tau)$. It is easy to show that the fHT is a unitary operator that commutes with translations and positive dilatations. In particular,    if $\{\psi_{j,k}\}_{j,k\in\Z}$ is a wavelet system, then $\bigl\{\mathcal{H}_{\tau}[\psi_{j,k}]\bigr\}_{j,k\in\Z}$ is also a wavelet system.

\begin{theorem}[Amplitude-phase representation of the {\DTCWT}, \cite{chaud1}]
Let $f$ be a function in $L^2(\R)$ with {\DTCWT} coefficients $c_j[k] = |c_j[k]|e^{i\omega_j[k]}$. Then
\begin{align}
\label{AmphPhase}
f &= \sum_{j,k\in\Z} |c_j[k]|  \ \Xi_{j,k}[w \; C(\omega_{j}[k])],
\end{align}
where, for $\omega\in [0,2\pi[$, $C(\omega)$ is given by
\begin{align*}
C(\omega)(x)= \cos(\omega_0 x + \xi_0 + \omega).
\end{align*}
\end{theorem}

Using this theorem, the authors of \cite{chaud1} provided new insights on the shiftability of the {\DTCWT}. Indeed,
Formula (\ref{AmphPhase}) gives an explicit interpretation of the phase parameter $\omega_j[k]$ as the phase-shift applied to the modulated sinusoid of the wavelet. More precisely, when $f$ is shifted over $h$, we get an amplitude-phase representation of the form
$$f^h = \sum_{j,k\in\Z} |c_j^h[k]| \Xi_{j,k}[w\, C(\omega_j^h[k])]$$
Hence, the localization window $w$ is kept fixed at scale $j$ while the oscillation is now shifted over $\omega_j^h[k]$ to better fit the underlying signal singularities/transitions.

We now extend upon their findings by employing the phase-compensated shift error to characterize the  shift errors $|e^{i2^j\omega_0h} c_j[k] -c_j^h[k]|$, as stated in the next proposition. This result will be applied in Theorem \ref{thm:shifterrordecay} to estimate the prediction error $c_j^h[k] \approx e^{i2^j\omega_0 h} c_j[k]$.

\begin{proposition}
\label{prop:ampphaseid}
Let $f$ be a real-valued function in $L^2(\R)$ with {\DTCWT} coefficients $c_j[k]=|c_j[k]|e^{i\omega_j[k]}$. Consider a translate $f^h = f(\cdot + h)$ with {\DTCWT} coefficients $c_j^h[k] = |c_j^h[k]|e^{i\omega_j^h[k]}$.
Then
\begin{equation*}
\begin{split}
\sum_{j\in J, k\in K} |e^{i2^j\omega_0 h} c_j[k] - c_{j}^h[k]|^2 &= \sqrt{\frac{\epsilon_1^2+\epsilon_2^2}{2}}
\end{split}
\end{equation*}
where $J,K\subseteq\Z$,
\begin{align*}
\epsilon_1 &= \left\|\sum_{j\in J, k\in K}|c_j[k]   |  \ \Xi_{j,k}\left[(w-w(\cdot+2^jh))C(\omega_j[k] + 2^j\omega_0 h) \right]\right\|_2
\intertext{and}
\epsilon_2 &=\left\|
 \sum_{j\in J, k\in K} |c_j[k]|  \ \Xi_{j,k}\left [(w-w(\cdot+2^jh))C(-\omega_j[k] -2^j\omega_0 h) \right] \right\|_2.
\end{align*}
\end{proposition}
\newpage
\begin{proof}
Let us first prove the theorem for $J=K=\Z$. Recall that $c_j[k] = \frac{1}{2}(a_j[k]-ib_j[k])$, where $a_j[k]$ and $b_j[k]$ are the coefficients of $f$ corresponding to the real and imaginary wavelet components respectively. The same formula for $\tilde{c}_j[k] = e^{i2^j\omega_0h}c_j[k]$ can be obtained by defining
\begin{align*}
\tilde{a}_j[k] &= 2 |c_j[k]|\cos(\omega_j[k] + 2^j\omega_0 h)
\intertext{and}
\tilde{b}_j[k] &= - 2 |c_j[k]|\sin(\omega_j[k] + 2^j\omega_0 h).
\end{align*}
 Now consider the functions $\tilde{f}_1$ and $\tilde{f}_2$, given by
$$\tilde{f}_1 = \sum_{j,k\in\Z} \tilde{a}_j[k] \psi_{j,k} \quad \text{and}\quad \tilde{f}_2 = \sum_{j,k\in\Z} \tilde{b}_j[k] \psi_{j,k}'.$$
Observe that
\begin{align*}
\sum_{j,k\in\Z}  |e^{i2^j\omega_0 h}c_j[k] - c_{j}^h[k]|^2 
&= \frac{1}{4}\|\tilde{f}_1 - f^h\|_2^2 + \frac{1}{4}\|\tilde{f}_2 - f^h\|_2^2 \\
&= \frac{1}{2}\left\|\frac{\tilde{f}_1 + \tilde{f}_2}{2} - f^h\right\|_2^2 + \frac{1}{2}\left\|\frac{\tilde{f}_1 - \tilde{f}_2}{2}\right\|_2^2
\end{align*}
where the last equality is a consequence of the parallelogram-law.
Since
\begin{align*}
\frac{1}{2}(\tilde{f}_1 + \tilde{f}_2) &= \sum_{j,k\in\Z} |c_{j}[k]| \cos(\omega_j[k] + 2^j\omega_0 h) \psi_{j,k}\\
&\quad\quad - \sum_{j,k\in\Z} |c_{j}[k]| \sin(\omega_j[k] + 2^j\omega_0 h) \psi_{j,k}' \\
&= \sum_{j,k\in\Z} |c_j[k]|  \ \Xi_{j,k}\left[w\; C(\omega_j[k] + 2^j\omega_0 h) \right]
\intertext{and}
f^h &= \sum_{j,k\in\Z}  \ |c_j[k]| \Xi_{j,k}[w(\cdot+2^j h) \; C(\omega_{j}[k]+ 2^j\omega_0 h)]
\end{align*}
we obtain that
\begin{align*}
\left\|\frac{\tilde{f}_1 + \tilde{f}_2}{2} - f^h\right\|_2 = \epsilon_1.
\end{align*}
On the other hand, we have the relations
\begin{align*}
\frac{1}{2}(\tilde{f}_1 - \tilde{f}_2) &= \sum_{j,k\in\Z} |c_j[k]|  \ \Xi_{j,k}\left[w\; C(-\omega_j[k] - 2^j\omega_0 h) \right]
\intertext{and}
\frac{1}{2}(f^h_1 - f^h_2) &= \sum_{j,k\in\Z} |c_j[k]|  \ \Xi_{j,k}\left[w(\cdot+ 2^j h) \; C(-\omega_j[k]-2^j\omega_0 h) \right]
\end{align*}
for\begin{align*}
f^{h}_1 = \sum_{j,k\in\Z} a_j^h[k] \psi_{j,k}\quad \text{and}\quad f^h_2=\sum_{j,k\in\Z}b_j^h[k]\psi'_{j,k}.
\end{align*}
 Note that  $f^h_1=f^h_2=f^h$ by definition. Hence, we can conclude that \begin{align*}
\left\|\frac{\tilde{f}_1 - \tilde{f}_2}{2}\right\|_2 &=
\left\|\frac{\tilde{f}_1 - \tilde{f}_2}{2}- \frac{\tilde{f}^h_1 - \tilde{f}^h_2}{2}\right\|_2= \epsilon_2.
\end{align*}
This proves the theorem for $J=K=\Z$.

For the general case, we replace the previous definitions of $\tilde{f}_1$, $\tilde{f}_2$, $f_1^h$ and $f^h_2$ by
\begin{align*}
\tilde{f}_1 &= \sum_{j\in J, k\in K} \tilde{a}_j[k] \psi_{j,k}\, ; &\tilde{f}_2 &=\sum_{j\in J, k\in K} \tilde{b}_j[k] \psi_{j,k}';\\
f^h_1 &= \sum_{j\in J, k\in K} a_j^h[k] \psi_{j,k}\, ;
&f^h_2 &= \sum_{j\in J, k\in K} b_j^h[k] \psi_{j,k}'.
\end{align*}
A similar calculation as before shows that \begin{align*}
\sum_{j\in J, k\in K}  &|e^{i2^j\omega_0h}c_j[k] - c_{j}^h[k]|^2\\
&= \frac{1}{4}\|\tilde{f}_1 - f^h_1\|_2^2 + \frac{1}{4}\|\tilde{f}_2 - f^h_2\|_2^2 \\
&= \frac{1}{2}\left\|\frac{\tilde{f}_1 + \tilde{f}_2}{2} - \frac{f^h_1 + f^h_2}{2}\right\|_2^2 + \frac{1}{2}\left\|\frac{\tilde{f}_1 - \tilde{f}_2}{2}- \frac{f^h_1 - f^h_2}{2}\right\|_2^2 \\
&= \frac{\epsilon_1^2 + \epsilon_2^2}{2} \tag*{\qedhere}
\end{align*}
\end{proof}

\noindent Proposition \ref{prop:ampphaseid} confirms the importance of a smooth localization window $w$ in order to minimize the {\DTCWT} shift error as previously indicated by Proposition \ref{prop:limWhEh}.

One way to measure the oscillatory behavior of $w$ is by using the concept of \emph{Lipschitz continuity}. By definition, $w$ is called an $\ell$-Lipschitz function if and only if $|w(y)-w(x)|\leq \ell |y-x|$ for every $x,y\in\R$. This leads us to the main theorem of this section, which is an immediate consequence of Proposition \ref{prop:ampphaseid}.

\begin{theorem}[Decaying rate of the {\DTCWT} shift error]
\label{thm:shifterrordecay}
Let $f$ be a real-valued function in $L^2(\R)$ with {\DTCWT} coefficients $c_j[k]$. Consider a translate $f^h = f(\cdot + h)$ with {\DTCWT} coefficients $c_j^h[k]$. If the localization window $w$ is a compactly supported $\ell$-Lipschitz function, so that $w-w(\cdot+2^jh)$ is zero outside $[p,q]$, then\begin{align}
\label{shifterrordecay}
\frac{|e^{i2^j\omega_0 h} c_j[k] - c_{j}^h[k]|}{|hc_j[k]|}
&\leq 2^j \ell(q-p)
\end{align}
for every $j,k\in\Z$.
\end{theorem}

Notice the formal analogy between the {\DTCWT} bound  in (\ref{shifterrordecay}) and the DWT translation sensitivity constraint in (\ref{translationsensitivity}).

\newpage
\section{Conclusion}

In this paper, we quantitatively investigated the shift error of the modulated dual-tree complex wavelet transform, which can be significantly reduced by performing a phase-compensation on the coefficients.

By introducing a formal description for the DWT translation sensitivity in Section 3, we were able to relate the phase-compensated shift error to the shift errors of the real and imaginary wavelet components. This study revealed that the superiority of the {\DTCWT} is attributed to the high overall frequency and the slowly varying localization window of the {\DTCWT}.
The improved shiftability is  particularly noticeable for significant coefficients.

In Section 4, we estimated the decaying rate of the phase-compensated shift error in case of orthonormal and modulated wavelet systems. This allows us to describe the prediction error in a similar way as the formal description of the DWT translation sensitivity.

\section*{Acknowledgment}
This research was supported by the Fund for Scientific Research Flanders (projects G.0206.08, G021311N and the Post-Doctoral Fellowship of Peter Schelkens) and by the Flemish Institute for the Promotion of Innovation by Science and Technology (IWT) (PhD bursary Adriaan Barri).

To conclude, the authors would like to thank the reviewer for his constructive ideas that greatly helped improving the paper.

\end{document}